\documentclass[11pt]{article}
\usepackage{fullpage}
\usepackage{latexsym}
\usepackage{amsfonts}
\usepackage{amssymb}
\usepackage{subeqnarray}

\def\Real{{I\!\!R}}

\newtheorem{definition}{Definition}[section]

\newtheorem{theorem}{Theorem}[section]
\newenvironment{proof}{{\it Proof. }}{\hfill $\Box$}
\newenvironment{remark}{{\it Remark. }}{\hfill $\triangleleft$ \medskip}

    \newcommand{\bea}{\begin{eqnarray}}
     \newcommand{\eea}{\end{eqnarray}}
     \newcommand{\bean}{\begin{eqnarray*}}
     \newcommand{\eean}{\end{eqnarray*}}
\def\ds{\displaystyle}

\def\RR{\mathbb R}

\def\[{\begin{equation}}
\def\]{\end{equation}}

\def\RR{\mbox{I}\!\mbox{R}}

\def\Box{\vrule width1ex height1ex}

\def\ds{\displaystyle}
\usepackage{multirow}
\newcommand{\EQ}{\begin{equation}\begin{array}{lllllllll}}
\newcommand{\EE}{\end{array}\end{equation}}
\newcommand{\MT}{\left( \begin{array}{ccccccccc}}
\newcommand{\EM}{\end{array}\right)}
\newcommand{\LM}{\left(\begin{array}{ccccccccc}}
\newcommand{\RM}{\end{array}\right)}
\newcommand{\LA}{\left\{ \begin{array}{ccccccccc}}
\newcommand{\RA}{\end{array}\right.}
\newcommand{\RAA}{\end{array}\right\}}

\title{\bf A 
characterization of normal forms for control systems}
\date{}
\author{\begin{tabular}{l@{\hspace{5mm}}l@{\hspace{5mm}}l} 
\multicolumn{3}{c}{} \\  
Boumediene Hamzi\footnote{Parts of this work were done while at Department of Mathematics, Duke University, Durham, NC 27708, USA. } & Jeroen S.W. Lamb & Debra Lewis \\
 Department of Mathematics & Department of Mathematics & Mathematics Department \\
Imperial College London & Imperial College London &   UC Santa Cruz \\ 
London, SW7 2AZ, UK & London, SW7 2AZ, UK &    Santa Cruz, CA 95064, USA  \\
b.hamzi@imperial.ac.uk & jeroen.lamb@imperial.ac.uk & lewis@ucsc.edu
\end{tabular}
}

\begin{document}
\jot 0.7em

\maketitle
%
%
\section{Introduction}

The study of the behavior of solutions of ODEs often benefits from deciding on a convenient choice
of coordinates. This choice of coordinates may be used to ``simplify'' the functional expressions that 
appear in the vector field in order that the essential features of the flow of the ODE near a critical point 
become more evident. In the case of the analysis of an ordinary differential equation in the neighborhood of
an equilibrium point, this naturally leads to the consideration of the possibility to remove the maximum
number of terms in the Taylor expansion of the vector field up to a given order. This idea was introduced by
H. Poincar\'e in \cite{poincare} and the ``simplified'' system is called normal form. There have been several
applications of the method of normal forms particularly in the context of bifurcation theory where one
combines between the method of normal forms and the center manifold theorem in order to classify bifurcations
\cite{guckenheimerholmes}. This approach was extended to control systems in continuous-time by Kang and Krener
(\cite{kang_and_krener1},  see also \cite{kang_and_krener} for a survey) and Tall and Respondek
(\cite{tall1}, see \cite{tall_survey} for a survey), and by Barbot et al. \cite{barbot} and Hamzi et al. in
discrete-time \cite{BH25,BH26}. The center manifold theorem was extended to control systems by Hamzi et al.
\cite{BH2,BH6} and combined with the normal forms approach to analyze and stabilize systems with
bifurcations in continuous and discrete-time \cite{BH8, BH17,BH18}.

On another side, even though in many textbook treatments (see eg \cite{guckenheimerholmes}) the emphasis is 
on the reduction of the number of monomials in the Taylor expansion, one of the main reasons for the success
of normal forms lies in the fact that it allows to analyze a dynamical system based on a simpler form and a
simpler form doesn't necessarily mean to remove the maximum number of terms in the Taylor series expansion.
This observation, led to introduce the so-called ``inner-product normal forms'' in \cite{belitskii,meyer1,
elphick}. They are based on properly choosing an inner product that allows to simplify the computations. This
inner-product will characterize the space overwhich one performs the Taylor series expansion. The elements in
this space are the ones that characterize the normal form. Our goal in this paper is to generalize such an
approach to control systems.

In section \S 2, we review some results about normal forms. In section \S 3, we develop a new method for deriving normal forms for control systems.

\section{Normal forms near equilibria of ODEs}\label{sec:nf}

In this section we briefly review some results on normal forms near equilibria of nonlinear ODEs.

Consider the nonlinear ODE in $\RR^n$ 
\[
\label{ode} 
\dot{x}=Ax+f(x),
\]
with $f \in {C}^{r+1}(\RR^n;\RR^n)$, $f(0)=0$ and $A=\frac{\partial f}{\partial x}|_{x=0}$ is in real or complex Jordan form.
Without loss of generality the latter condition can be met by application of a linear coordinate transformation.

The goal is to find a change of coordinates 
\[
\label{chancor} 
x=\xi(y),
\]  with $\xi \in C^r(\RR^n;\RR^n)$ in a neighborhood  of the origin,
such that the Taylor expansion of (\ref{ode}) is simple, making essential features of the flow
 of (\ref{ode}) near the equilibrium $x=0$ more evident. The desired simplification of (\ref{ode}) will be
obtained, up to terms of a specified order, by constructing a near identity  coordinate transformation from a
sequence of compositions of coordinate transformations of the form  
(\ref{chancor}) with
\[
\label{chancor1}
\xi(y)=\exp(\xi^{[k]})(y)=y+\xi^{[k]}(y)+O(|y|^{k+1}),
\]
where $y\in\RR^n$ is close to zero, $\xi^{[k]}\in H_n^k$ ($k \ge 2$), the vector space of
 homogeneous polynomials of degree $k$ in $n$ variables with values in $\RR^n$, and 
 $\exp(\xi^{[k]})$ denotes the time-one flow of the ODE $\dot{y}=\xi^{[k]}(y)$. 
 We consider a formal power series expansion of $f$ in (\ref{ode}) and write
 \[
 f(x)=f^{[2]}(x)+f^{[3]}(x)+\ldots,
 \]
with $f^{[k]}\in H_n^k$. From (\ref{chancor1}) we obtain
\[\label{chancor2} \xi^{-1}(y)=y-\xi^{[k]}(y)+O(|y|^{2k}). 
\]
Substituting  (\ref{chancor}), (\ref{chancor1}) and (\ref{chancor2}) in (\ref{ode}), we get
\[
\label{y_eqn} \dot y={ A} y+\cdots+f^{[k-1]}(y)
+
f^{[k]}(y)-(L_{A} \xi^{[k]})(y)
+O(|y|^{k+1}),
\]
with the Lie derivative $L_{A}$ defined on vector fields $f$ as
\[
(L_{ A}f)(y):=\frac{\partial f(y)}{\partial y}{A}y-{A}f(y).
 \]
 In the present context $L_A$ is also known as the \emph{homological} operator.

The Lie derivative leaves $H_n^k$ invariant, $L_{\cal A}:H_n^k\to H_n^k$. We denote its range in $H_n^k$ as 
${\mathcal{R}}^k$ and let ${\mathcal{C}}^k$ denote a complement  of
${\mathcal{R}}^k$ in $H_n^k$ 
\[
\label{eqn:decompesvec}
H_n^k={\mathcal{R}}^k \oplus {\mathcal{C}}^k, \quad k\ge 2. 
\]
We define a \emph{normal form} of $f$ of order $r$ as a Taylor expansion of the vector field with linear part and terms  $f^{[k]}\in{\mathcal{C}}^k$ for $2\leq k\leq r$. 
We may associate the choice of complement ${\mathcal{C}^k}$ to an inner product on $H_n^k$, for which it is the \emph{orthogonal} complement of ${\mathcal{R}}^k$ in $H_n^k$, i.e.\,
$\mathcal{C}^k:=({\mathcal{R}}^k)^{\perp}.$

A convenient choice of inner product was introduced by Belitskii \cite{belitskii}, Meyer \cite{meyer1} and Elphick \emph{et al.} \cite{elphick}, enabling the characterization of expression of ${\mathcal{C}^k}$ as the kernel of the Lie derivative of $A^\ast$ (the adjoint of linear part $A$ of the vector field at the equilibrium).
Denoting monomials in shorthand notation as $x^\ell:=x_1^{\ell_1}\cdots x_n^{\ell_n}$ with $\ell!:=\ell_1!\cdots\ell_n!$, we define an inner product on polynomials
%
\[ \label{innerp}
p(x)=\sum_{\ell}p_{\ell}x^{\ell}, \quad q(x)=\sum_m q_m x^m,~~\mbox{as}~~
\langle p,q \rangle=\sum_m m! p_
m q_m.
\end{equation}
For vector polynomials we define the corresponding inner product as the sum of the inner products between the polynomials of corresponding vector components. 
The inner product (\ref{innerp}) with $T\in gl(n, \RR)$ and $T^{\ast}$ denoting its adjoint (with respect to
 the standard inner product on $\RR^n$) satisfies \cite{belitskii,elphick} 
\[
\langle p \circ T, q\rangle=\langle p, q\circ T^{\ast}\rangle. 
\]
Accordingly, one obtains that the adjoint of $L_A$ on $H_n^k$ with the above defined inner product
 satisfies the following
relation \cite{belitskii,elphick}
\begin{equation}\label{adjoint}
(L_A)^{\ast}=L_{A^{\ast}}.
\end{equation}
By application of the Fredholm alternative, it follows that 
$({{\mathcal{R}}^k})^{\perp}=\ker(L_{A}^{\ast}|_{H_n^k})$. In combination with (\ref{adjoint}), this leads us
to $$\mathcal{C}^k= \ker(L_{A^{\ast}}|_{H_n^k}),$$ as a result of which nonlinear elements of the normal form
$g$ satisfy the linear PDE \[L_{A^{\ast}}g=0.\]
This  PDE can be solved explicitly using the method of  characteristics (for more details on this method, see for example \cite{courant}). 

We recall that since $L_{A^{\ast}}$ is a Lie derivative, it follows that the nonlinear elements of the normal form commute with the group
\[
G=\overline{\{\exp(A^{\ast}t)~|~t\in\RR\}}.
\]


We finally note that \[\ker(L^k_{A^{\ast}})=\ker(L^k_{A_s^{\ast}})\cap \ker(L^k_{A_n^{\ast}}),\] where
$A=A_s+A_n$ is the Jordan-Chevalley decomposition of $A$ in its (mutually commuting) semi-simple and nilpotent
parts.	As $A^\ast$ commutes with $A_s$ but not with $A_n$ (if nonzero), if $A$ is not semi-simple, only a
subgroup of $G$ (as defined above) is a symmetry group of the normal form. 
In general, with the above choices made, the normal form is equivariant with respect to the group 
\[
G_s=\overline{\{\exp(A^{\ast}_st)~|~t\in\RR\}}.
\]
The appearance of this symmetry group is an important feature.

\section{Normal Forms of Nonlinear Control Systems} \label{sect_control}

The object of this section is to extend the normal form theory set out above to nonlinear control systems. 
We consider the nonlinear control system 
\[
\label{sysgen}
\dot{x}=f(\tilde{x}),
\] 
with $\tilde{x}=(x,u)^T \in  \RR^n \times \RR^m$ and $u\in\RR^m$ representing the control. 
In lowest linear order Taylor expansion, the control system takes the form 
$$
\dot{x}={\cal A}\tilde{x} 
+  O(|\tilde{x}|^2),
$$
with ${\cal A}:=\left(\begin{array}{cc} A & B \end{array} \right)$ and $A:=\frac{\partial f(\tilde{x})}{\partial x}|_{\tilde{x}=0}$, $B:=\frac{\partial f(\tilde{x})}{\partial u}|_{\tilde{x}=0}$ .  

We consider the effect of coordinate transformations of the form
\[
\label{chancort}
\tilde{x}=p(\tilde{y})
=\exp(p^{[k]})(\tilde{y})= \tilde{y}+p^{[k]}(\tilde{y})+O(|\tilde{y}|^{k+1}),
\]
where $\tilde{y}=(y,v)^T$ and $p^{[k]}\in S^k_{n+m,n}$ with
\[
S^k_{n,m}:=\{(p^{[k]}_x, p^{[k]}_u)^T~|~ p^{[k]}_x\in H_n^k,~p^{[k]}_u\in H_{n+m,m}^k\},
\] 
with $H_{n+m,m}^k$ denoting the vector space of homogeneous polynomials of degree $k$ from $\RR^{n+m}$ to $\RR^m$.
The skew product form of $p^{[k]}\in S^k_{n,m}$,  $p^{[k]}(\tilde{y})=(p^{[k]}_x(y), p^{[k]}_u(\tilde{y}))^T$,  guarantees that the control system is transformed to another control system of the same type. If $p_x$ would
depend on $u$ then the coordinate transformation would introduce a relationship involving $\dot{u}$. We obtain
\[
\dot y=\mathcal{A} \tilde{y}+\cdots+f^{[k-1]}(\tilde{y})
+f^{[k]}(\tilde{y})-(\mathcal{L}_{\mathcal{A}} p^{[k]})(\tilde{y})
+O(|\tilde{y}|^{k+1}),
\]
where the homological operator $\mathcal{L}_{\mathcal{A}}: S^k_{n,m}\to H^k_{n+m,n}$  has the form
\[\label{homeqcontrol}
({\cal L}_{\cal A}p^{[k]})(\tilde{y})=Dp_x^{[k]}(y){\cal A}\tilde{y}-{\cal A} p^{[k]}(\tilde{y})=
(L_A p_x^{[k]})(y)+Dp_x^{[k]}(y)Bu-Bp_u^{[k]}(\tilde{y}).
\]
We recognize in this expression the Lie derivative $L_A$, that is equal to  ${\mathcal L}_{\mathcal A}$ in case $B=0$. Indeed, $f(y,0)$ (the part of $f$ that does not depend on $u$) can be put into a
$G_s$-equivariant  normal form, using coordinate transformations of the form $p(\tilde{y})=(\exp(p_x)(y),u)$ only. 

We now proceed to characterize a normal form by a (choice of) complement of the range 
of $\mathcal{L}_\mathcal{A}$. In order to do so in analogy to the theory developed for ODEs, we temporarily
take the viewpoint as if the coordinate transformation would be for the ODE $(\dot{x}, \dot{u})=(f(x,u),h(x,u))$,
for some $h:\Real^{m+n}\to\Real^m$ with $D_xh(0,0)=0$ and $D_uh(0,0)=0$.
The homological operator for the latter ODE, with coordinate transformations of the form 
(\ref{chancort}) takes precisely the form of the Lie derivative $L_{\mathcal{A}_0}$, with 
$$
\mathcal{A}_0=\left( \begin{array}{cc}A& B\\0&0\end{array}\right),
$$ 
so that $L_{\mathcal{A}_0}=(\mathcal{L}_\mathcal{A},0)$ and  $\mathcal{L}_\mathcal{A}:=\pi
L_{\mathcal{A}_0}$, 
with $\pi:\Real^{m+n}\to\Real^n$ denoting the canonical projection $\pi(x,u):=x$.

 We may thus choose the
complement  of the range of $\mathcal{L}_\mathcal{A}$ as the projection under $\pi$ of the orthogonal
complement $\mathcal{C}^k$ to the range of $L_{\mathcal{A}_0}$ taken with respect to the inner product
(\ref{innerp}) with $x_{n+i}=u_i$, $i=1,\ldots, m$, i.e.
\[
\mathcal{C}^k:=
\{q\in H^k_{m+n}~|~ \langle q, L_{\mathcal{A}_0}p\rangle=0,~\forall p\in S^k_{n,m}\}.
\]
By the Fredholm alternative we have
\[
 \langle q, L_{\mathcal{A}_0}p\rangle=
  \langle L_{\mathcal{A}_0^*} q, p\rangle,
\]
so that this complement takes the form
\[
\mathcal{C}^k:=
\{ q\in H^k_{m+n}~|~ L_{\mathcal{A}_0^*}q\in (S^k_{n,m})^\perp\}.
\]
By the definition of the inner product (\ref{innerp}), 
$$(S^k_{n,m})^\perp=
\{ q\in H^k_{m+n}~|~ q(x,0)=0\},$$ 
i.e.\, the subset of vector  polynomials in $H^k_{m+n}$ for which each constituting monomial contains a factor $u_i$, $i=1,\ldots, m$.
The complement to the range of $\mathcal{L}_\mathcal{A}$ characterising the corresponding normal form is $\pi\mathcal{C}^k$. By writing out the relevant operators, the following result follows immediately.


 \begin{theorem}[Control normal form] 
Consider a finite order in Taylor expansion of the vector field defining the control system (\ref{sysgen}), $$f(\tilde{x})={\cal A}\tilde{x} +\sum_{k=2}^N f^{[k]}(\tilde{x})+O(|\tilde{x}|^{k+1}), \mbox{~with~}f^{[k]}\in H_{m+n,n}^k.$$ 
By a choice of coordinates, the nonlinear parts $f^{[k]}$  can be made to satisfy
\[\label{nfeqn} \hat\mathcal{L}_{\mathcal{A}^*}{f}^{[k]}(x,0)=0\] where
\[\label{pdehom}
\hat\mathcal{L}_{\mathcal{A}^*}{{f}^{[k]}}(\tilde{x}):=D_{\tilde{x}} f^{[k]}(\tilde{x}) {\cal A}^* x - A^* f^{[k]}(\tilde{x}),
\]
and $\tilde{x}=(x,u)$.
\end{theorem}
\begin{remark}
We note that by restricting first to coordinate transformations that do not involve $u$, we can achieve $G_s$-equivariance of the control system to any desired order. Then we can refine the normalization further using $G_s$-equivariant coordinate transformations that preserve this equivariance.
\end{remark}


\section{Illustrations}
\subsection{Linearly Controllable Case}

To illustrate this method, consider the nonlinear control system  $\Sigma$ in (\ref{sysgen}) with one input, i.e. $m=1$, and assume that its linearization is controllable. 
From linear  control theory we know that there exists a linear change of coordinates and feedback that allows to transform  the linear part  in the Brunovsk\`y form, i.e.  

\begin{eqnarray} \label{matr}
A&=&\left(\begin{array}{ccccc}0 & 1 & 0 & \cdots & 0\\ 0 & 0 & 1
& \cdots & 0\\ \vdots & \vdots & \vdots & \ddots & \vdots \\ 0 & 0
& 0& \cdots & 1 \\ 0 & 0 & 0 &\cdots & 0
\end{array}
\right),\; B= \left(
\begin{array}{c}0 \\ 0 \\ \vdots \\ 0 \\ 1
\end{array}
\right).
\end{eqnarray}

In this case, the PDE (\ref{pdehom}) becomes
\[\label{pde_lincont} \left\{\begin{array}{rcl}
\ds x_1 \frac{\partial p_1}{\partial x_2}+\cdots+x_{n-1} \frac{\partial p_1}{\partial x_n}+x_n  \frac{\partial p_1}{\partial u}&=& 0\\
\ds x_1 \frac{\partial p_2}{\partial x_2}+\cdots+x_{n-1} \frac{\partial p_2}{\partial x_n}+x_n  \frac{\partial p_2}{\partial u}-p_1&=& 0\\
&\vdots& \\
\ds x_1 \frac{\partial p_n}{\partial x_2}+\cdots+x_{n-1} \frac{\partial p_n}{\partial x_n}+x_n  \frac{\partial p_n}{\partial u}-p_{n-1}&=& 0
\end{array}\right. \]
that we'll solve using the method of characteristics. 

\begin{theorem}
Consider the nonlinear control system $\Sigma$ given by (\ref{sysgen}). There exist a change of coordinates and feedback (\ref{chancort}) such that $\Sigma$ writes as
\[
\begin{array}{rcl} 
\ds \dot{x}_1&=& \ds x_2+\Phi_1(\ell_1,\cdots,\ell_{r+1}),\\
\ds \dot{x}_2&=& \ds x_3+\Phi_2(\ell_1,\cdots,\ell_{r+1})+\int p_1(x,u) \frac{dx_2}{x_1},\\
&\vdots& \\
\ds \dot{x}_n&=& \ds u+ \Phi_n(\ell_1,\cdots,\ell_{r+1})+\int p_{n-1}(x,u) \frac{dx_2}{x_1},
\end{array}
\]
with  $\Phi_i(\ell_1,\cdots,\ell_n)$   are functions satisfying \begin{eqnarray} \label{conditions}
 \frac{\Phi_i(\ell_1,\cdots,\ell_n)}{x_1^p}\bigg{|}_{x_1=0}=0 & \mbox{for} & p=0,\cdots,n-i, \\
\end{eqnarray}
and $\ell_1(x)=x_1$, $\ds \ell_2(x)=\frac{x_2^2}{2}-x_1x_3$, $\cdots$,  
$\ds \ell_{i}(x)=\frac{1}{2}x_i^2+\sum_{k=1}^{n-p}(-1)^k x_{i-k}x_{i+k}$ for $i=2,\cdots,r+1$,  and $r$ is such that $r=n/2$ if $n$ is even, and  $r=(n-1)/2$ if $n$ is odd  (here, $x_0=0$ and  $x_{n+1}=u$)
\end{theorem}

\begin{proof}

In the $n-$dimensional space of the variables $x_1$, $x_2$,$\cdots$,$x_n$ we determine the curves $x_i=x_i(s)$ in terms of a parameter $s$ by means of the system of
ordinary differential equations that represent the characteristic curves
\[\label{characteristics} \left\{\begin{array}{rcl}
  \ds \frac{d x_1}{d s}&=&0 \\
\slabel{dt_eqn} \ds \frac{d x_2}{d s}&=&x_1 \\
&\vdots& \\
\ds   \frac{d x_n}{d s}&=&x_{n-1} \\
 \ds \frac{d u}{d s}&=&x_{n}
\end{array}\right. \]

Along the characteristic curves and using the chain rule, the systems of PDEs (\ref{pdehom}) writes as
\[\label{peqn_lincont}\begin{array}{rclll}
  \ds \frac{d p_1}{d s}&=& \ds \frac{d x_1}{d s} \frac{\partial p_1}{\partial x_1}+  \frac{d x_2}{d s} \frac{\partial p_1}{\partial x_2}+\cdots+ \frac{d x_n}{d s} \frac{\partial p_1}{\partial x_n}+\frac{d u}{d s}  \frac{\partial p_1}{\partial u}&=&0\\
\ds  \frac{d p_2}{d s}&=& \ds \frac{d x_1}{d s} \frac{\partial p_1}{\partial x_1}+  \frac{d x_2}{d s} \frac{\partial p_1}{\partial x_2}+\cdots+ \frac{d x_n}{d s} \frac{\partial p_1}{\partial x_n}+\frac{d u}{d s}  \frac{\partial p_1}{\partial u}&=&p_1  \\
&\vdots& \\
\ds \frac{d p_n}{d s}&=& \ds \frac{d x_1}{d s} \frac{\partial p_1}{\partial x_1}+  \frac{d x_2}{d s} \frac{\partial p_1}{\partial x_2}+\cdots+ \frac{d x_n}{d s} \frac{\partial p_1}{\partial x_n}+\frac{d u}{d s}  \frac{\partial p_1}{\partial u}&=&p_{n-1} 
\end{array}\]
Hence, along the characteristic curves defined by (\ref{characteristics}),  the  systems of PDEs (\ref{pdehom}) transforms into a set of ODEs
\[\label{peqn_odes}\begin{array}{rcl}
\ds  \frac{d p_1}{d s} &=& \ds 0 \\
 \ds \frac{d p_2}{d s}  &=& \ds p_1 \\
&\vdots& \\
 \ds \frac{d p_n}{d s} &=& \ds p_{n-1}\end{array}
\]
This system of ODEs can be solved explicitly 
$$\begin{array}{rcl}
\ds   p_1(s) &=& \ds c_1 \\
 \ds p_2(s)  &=& \ds c_2+\int p_1(s) ds \\
&\vdots& \\
 \ds  p_n(s) &=& \ds c_n + \int p_{n-1}(s)ds \end{array}
$$

The ``constants of integration'', $c_i$, are the constants along the characteristic curves which are the trivial first integrals of the system (\ref{characteristics}). One can check that
they are given by $\ell_1(x)=x_1$, $\ds \ell_2(x)=\frac{x_2^2}{2}-x_1x_3$, $\cdots$,  
$\ds \ell_{i}(x)=\frac{1}{2}x_i^2+\sum_{k=1}^{n-p}(-1)^k x_{i-k}x_{i+k}$ for $i=2,\cdots,r+1$,  and $r$ is such that $r=n/2$ if $n$ is even, and  $r=(n-1)/2$ if $n$ is odd  (for notation convenience, $x_0=0$ and  $x_{n+1}=u$). 

From
(\ref{dt_eqn}), we have\footnote{We can also use $ds=\frac{dx_{i+1}}{x_i}$ and in this case the normal form will be parametrized by $x_{i+1}$. We can also parameterize each component with a different parameterization. }  $ds=\frac{dx_2}{x_1}$, and 
the solution of (\ref{peqn_odes}) is given by
\[\label{psln_lincont}\begin{array}{rcl}
p_1(x,u) &=& \Phi_1(\ell_1,\cdots,\ell_{r+1})\\
p_2(x,u) &=& \Phi_2(\ell_1,\cdots,\ell_{r+1})+\int p_1(x,u) \frac{dx_2}{x_1} \\
\vdots &=& \vdots \\
p_n(x,u) &=& \Phi_n(\ell_1,\cdots,\ell_{r+1})+\int p_{n-1}(x,u) \frac{dx_2}{x_1} 
\end{array}\]
where $\Phi_i(\ell_1,\cdots,\ell_n)$, $i=1,\cdots,n$, are functions of the variables $\ell_1,\cdots,\ell_n$ and are thus constants along the characteristic curves define in (\ref{characteristics}). 
Since $\Phi_i(\ell_1,\cdots,\ell_n)$ and $\tilde{q}_i(x,u)$ satisfy the conditions
\begin{eqnarray*} 
 \frac{\Phi_i(\ell_1,\cdots,\ell_n)}{x_1^p}\bigg{|}_{x_1=0}=0 & \mbox{for} & p=0,\cdots,n-i, 
\end{eqnarray*}
thus $p_1,\cdots,p_{n-1}$   are divisible by $x_1$. Hence $p_1(x,u),\cdots p_n(x,u)$ in (\ref{psln_lincont}) are polynomials.


\end{proof}

\subsubsection{Example} Consider a two dimensional system with controllable linearization. In this case, the linear part of (\ref{sysgen}) writes as \begin{eqnarray}\dot{x}_1&=&x_2 \\ \dot{x}_2&=&u \end{eqnarray}
and the PDE (\ref{pdehom}) writes as
\[
\begin{array}{ccc}
\ds x_1 \ds \frac{\partial p_1}{\partial x_2}+x_2 \ds \frac{\partial p_1}{\partial u} &=& 0\\ 
\ds x_1 \ds \frac{\partial p_2}{\partial x_2}+x_2 \ds \frac{\partial p_2}{\partial u} &=& 0\\ 
 \end{array}
\]
 Hence, we get
\begin{subeqnarray}\slabel{pde1_dim2}
\frac{dp_1}{ds}&=&0,\\
\slabel{pde2_dim2}\frac{dp_2}{ds}-p_1&=&0
\end{subeqnarray}
with
\begin{subeqnarray}
\frac{dx_1}{ds}&=& 0 \\
\slabel{dtdx2}\frac{dx_2}{ds}&=& x_1 \\
\slabel{dtdu} \frac{du}{ds}&=& x_2 \\
\slabel{p1dim2} \frac{dp_1}{ds}&=& 0\\
\slabel{p2dim2} \frac{dp_2}{ds}-p_1&=& 0
\end{subeqnarray}

We thus deduce the following parametrization of the solution
\begin{subeqnarray}
\slabel{transform1} x_1 &=& x_{1,0} \\
\slabel{transform2} x_2 &=& x_{1,0}s+x_{2,0} \\
\slabel{transform3} u &=& \frac{x_{1,0}}{2} s^2+x_{2,0} s+x_{3,0}
\end{subeqnarray}

The first integrals are $\ell_1(x,u)=x_1$ and $\ell_2(x,u)=2x_1u-x_2^2$. From (\ref{p1dim2})-(\ref{p2dim2})  we deduce that
\[p_1(x,u) =  \Phi_1(x_1,2x_1u-x_2^2)\]
\[p_2(x,u) = \int p_1(t)  dt+\Phi_2(x_1,2x_1u-x_2^2)\]
We can use either (\ref{dtdx2}) or (\ref{dtdu}) to express the normal form as a function of $x_2$ or $u$. For example, using (\ref{dtdx2}) we deduce that
$\ds dt=\frac{dx_2}{x_1}$. Moreover, using (\ref{conditions}), we obtain conditions on $\Phi_i(\ell_1,\ell_2)$ and $\tilde{q}_i$, $i=1,2$,
\begin{subeqnarray*}
\Phi_1(\ell_1,\ell_2)|_{x_1=0}&=&0, \\
\end{subeqnarray*}

At the quadratic level these conditions imply that 
\begin{subeqnarray}
\Phi_1(\ell_1,\ell_2)|_{x_1=0}&=& \phi_{11}x_1^2+O(x,u)^3 \\ 
\Phi_2(\ell_1,\ell_2)|_{x_1=0}&=& \tilde{\phi}_{11}x_1^2+\tilde\phi_{12}(2x_1u-x_2^2)+O(x,u)^3
\end{subeqnarray}

Hence
\begin{subeqnarray}
p_1(x,u)&=&\phi_{11}x_1^2+O(x,u)^3 \\
p_2(x,u)&=&\phi_{11}x_1 x_2+\tilde\phi_{11}x_1^2+\tilde\phi_{12}(2x_1u-x_2^2)+O(x,u)^3 
\end{subeqnarray}

Hence the normal form has the form
\begin{subeqnarray}
\dot{x}_1&=&x_2+\phi_{11}x_1^2+O(x,u)^3 \\
\dot{x}_2&=&u+\phi_{11}x_1 x_2+\tilde\phi_{11}x_1^2+\tilde\phi_{12}(2x_1u-x_2^2)+O(x,u)^3  
\end{subeqnarray}

\subsection{Systems with Uncontrollable Linearization}

Now, consider the nonlinear control system  $\Sigma$ in (\ref{sysgen}) with one input, i.e. $m=1$, and assume that the system has $r$ uncontrollable modes. 
From linear  control theory we know that there exists a linear change of coordinates and feedback that allows to write the linear part as
\begin{eqnarray}
 \dot{z}&=&A_1z+O(z,x,u)^2,\\
\dot{x}&=&A_2x+B_2u+O(z,x,u)^2
\end{eqnarray}
%
where $z \in \RR^{r \times 1}$, $x \in \RR^{(n-r) \times 1}$, $A_1 \in \RR^{r \times r}$, and $(A_2,B_2)\in \RR^{(n-r) \times (n-r)} \times \RR^{(n-r)\times 1}$ are in the Brunovsk\`y form.

In this case, ${\cal A}=\left(\begin{array}{rcl}
A_1 & 0\\ A_2 & B_2 
\end{array}\right)$,  $\tilde{x}=(z,x,u)^T$ in the PDE (\ref{pdehom}). Let's note that when $r=0$ we recover the case in the preceding section and we can find a general explicit solution. However, when $r \ne 0$ a general solution is not as easily found and depends on $A_0$. We'll illustrate the method through an example.

\subsubsection{Example}
Consider the system whose linear part writes as
\[\left\{\begin{array}{rcl} \dot{z}&=&O(z,x,u)^2,\\ \dot{x}_1&=&x_2 +O(z,x,u)^2\\ \dot{x}_2&=&u+O(z,x,u)^2 \end{array}\right. \]
This system has uncontrollable linearization and the uncontrollable dynamics corresponds to the $z-$dynamics.

The elements of the normal form satisfy the PDE
\[\label{pde_dim2}\left\{\begin{array}{rcl} \ds z\frac{\partial p_1}{\partial x_2} + x_2 \frac{\partial p_1}{\partial u} &=&0\\ 
  \ds z\frac{\partial p_2}{\partial x_2} + x_2 \frac{\partial p_2}{\partial u} &=&0\\ 
  \ds z\frac{\partial p_3}{\partial x_2} + x_2 \frac{\partial p_3}{\partial u}-p_2 &=&0
  \end{array} \right.\]
The equation of the characteristics is
\[\left\{\begin{array}{rcl} \ds \frac{d p_1}{d s}  &=& 0\\ 
 \ds \frac{d p_2}{d s}  &=&0 \\ 
  \ds \frac{d p_3}{d s}  &=& p_2
  \end{array} \right.\]
The characteristic equations are $z=c_1$, $x_1=c_2$, $x_2=c_1 s+c_3$, $u=\frac{c_1}{2}s^2+c_3s+c_4$.  We can either parametrize by $x_2$ or $u$ 
by writing $ds =\frac{dx_2}{z}$ or $ds=\frac{du}{x_2}$.

The solution of the system of PDEs (\ref{pde_dim2}) is 
\[\begin{array}{rcl}
p_1&=&\Psi_0(z,x_1,x_2^2-2zu)\\
p_2&=&\Psi_1(z,x_1,x_2^2-2zu)\\
p_3&=&\Psi_2(z,x_1,x_2^2-2zu)+\int p_2(z,x,u) \frac{dx_2}{z} 
\end{array}
\]
The normal form is thus given by
\[\begin{array}{rcl}
\dot{z}&=&\Psi_0(z,x_1,x_2^2-2zu)\\
\dot{x}_1&=&x_2+\Psi_1(z,x_1,x_2^2-2zu) \\
\dot{x}_2&=&u+\Psi_2(z,x_1,x_2^2-2zu)+\int p_2(z,x,u)\frac{dx_2}{z} 
\end{array}
\]



\section{Concluding remark and future extensions:} 

Given the preceding, one could think about hyper normal forms where instead of normalizing with respect to the linear term, one normalizes the quadratic term with respect to the linear term, then normalize the cubic term with respect to the sum of the linear and quadratic terms, and so forth.  This direction has been fruitful for systems without control \cite{murdock, murdock1} and its extension to the control case is the object of future research. Several other extensions are possible for this work. One could think about characterizing completely the normal form in the case of systems with uncontrollable linearization, developing the Hamiltonian case, and computing the coefficients in the normal form directly from the original system.
\bigskip

{\bf Acknowledgement:} The first author is thankful to Prof. Murdock for useful comments and for pointing out to references \cite{belitskii1, murdock1}. BH also thanks the European Union for financial support received through an International Incoming Marie Curie Fellowship.

\end{document}